\documentclass[12pt]{amsart}
\usepackage{amssymb, amsfonts}

\newcommand{\<}{\langle}
\renewcommand{\>}{\rangle}

\renewcommand{\th}{\theta}

\newcommand{\cA}{{\mathcal A}}
\newcommand{\cB}{{\mathcal B}}

\newcommand{\cH}{{\mathcal H}}
\newcommand{\cK}{{\mathcal K}}
\newcommand{\cL}{{\mathcal L}}

\newcommand{\cS}{{\mathcal S}}

\newcommand{\cV}{{\mathcal V}}
\newcommand{\cW}{{\mathcal W}}

\newtheorem{theorem}{Theorem}[section]
 
\newtheorem{proposition}[theorem]{Proposition} 

\theoremstyle{definition}   
\newtheorem{definition}[theorem]{Definition}

\begin{document}

\title
{Concrete realizations of quotients of operator spaces}
\author{Marc A. Rieffel}
\address{Department of Mathematics \\
University of California \\ Berkeley, CA 94720-3840}
\email{rieffel@math.berkeley.edu}
\thanks{The research reported here was
supported in part by National Science Foundation grant DMS-0753228.}


\begin{abstract}
Let  $\cB$ be a unital C*-subalgebra 
of a unital C*-algebra $\cA$, so that $\cA/\cB$ is an abstract
operator space. We show how to realize $\cA/\cB$ as a concrete
operator space by means of a completely contractive map
from $\cA$ into the algebra of operators on a Hilbert space,
of the form $A \mapsto [Z, A]$ where $Z$ is a Hermitian unitary
operator. We do not use Ruan's theorem concerning concrete
realization of abstract operator spaces. 
Along the way we obtain corresponding results for abstract
operator spaces of the form $\cA/\cV$ where $\cV$ is a
closed subspace of $\cA$, and then for the more special
cases in which $\cV$ is a $*$-subspace or an operator system.
\end{abstract}

\maketitle
\allowdisplaybreaks

\section*{Introduction}

In a recent paper \cite{R24} I showed that if 
 $\cB$ is a unital C*-subalgebra 
of a unital C*-algebra $\cA$ and if $L$ is the quotient norm on $\cA/\cB$ pulled back to $\cA$,
that is,
\[
L(A) = \inf \{\|A-B\|\  :  B \in \cB\}  
\]
for $A \in \cA$, then there is a unital $*$-representation $(\cH, \pi)$ of $\cA$
and a Hermitian unitary operator $U$ on $\cH$
such that 
\[
L(A) = \|[U, \pi(A)]\|
\]
for all $A \in \cA$. The consequence of this that most interested me
is that it follows that $L$ satisfies the Leibniz 
inequality
\[
L(AC) \leq L(A)\|C\| +\|A\|L(C)
\]
for all $A, C \in \cA$. But another interesting consequence is
that the map $A \mapsto [U, \pi(A)]$ gives an isometry of
$\cA/\cB$ into $\cL(\cH)$. Now $\cA/\cB$ is actually an
operator space, in the sense of having a compatible family
of norms on all the matrix spaces over it (reviewed below),
and this suggests that one should seek a natural construction of a 
``complete isometry'' from $\cA/\cB$
into the algebra of operators on some Hilbert space (i.e. one
respecting the norms on all the matrix spaces). The main purpose
of this article is to provide such a construction. In fact, we show
that there exists a complete isometry that is again of the form
 $A \mapsto [U, \pi(A)]$. As a consequence we obtain a
``matrix Leibniz seminorm'' on $\cA$ by taking the norms of
the commutators.

Now matrix Leibniz seminorms played a crucial role in my earlier
paper \cite{R17} relating vector bundles and Gromov-Hausdorff
distance, and one of my projects is to generalize the main results
of that paper to the setting of non-commutative C*-algebras, so that
they can be applied, for example, to the setting of quantizations
of coadjoint orbits that I studied in \cite{R7, R21, R25}. 
(Matrix seminorms have already been defined and discussed in
this context in \cite{Wuw1, Wuw2, Wuw3},  but the Leibniz property was not used there.) 
Actually, for infinite-dimensional C*-algebras, 
the C*-metrics as defined in \cite{R21}
are discontinuous and only densely defined. But they are required
to be lower semi-continuous with
respect to the C*-norm, and in all of the examples that I know of one proves that they are
lower semi-continuous by showing that they are the supremum of an infinite
family of continuous Leibniz seminorms. This provides ample reason for
studying continuous Leibniz seminorms.
Thus the results of the present paper provide some interesting information about matrix Leibniz seminorms, 
and so provide a small step forward in my project.

We will actually develop some of our results in a more general context,
namely that in which $\cV$ is a closed subspace of a unital C*-algebra
$\cA$, so that $\cA/\cV$ is an abstract operator space. We show
that in this case there exists a
unital $*$-representation $(\cH, \pi)$ of $\cA$,
and projections $P$ and $Q$ on $\cH$, such that
the linear mapping $\Psi$ from $\cA$ to $\cL(\cH)$ defined
by
\[
\Psi(A) = Q\pi(A)P
\]
gives a complete isometry from $\cA/\cV$ into $\cL(\cH)$. 
To show this we do not
need to use Ruan's construction \cite{ER} of 
complete isometries from abstract
operator spaces into operator algebras (essentially because 
C*-algebras
can be considered to be concrete operator spaces, by the
Gelfand-Naimark theorem). In fact, our results immediately apply
to the situation of a concrete operator space $\cW$ and a
closed subspace $\cV$ of $\cW$, so that $\cW/\cV$ is an
abstract operator space, just by considering the unital C*-algebra
$\cA$ generated by the concrete operator space $\cW$. All of
this is discussed in Section \ref{secos}.

Now a C*-subalgebra is in particular a $*$-subspace. For this
reason we discuss in Section \ref{secstar} the situation in
which $\cV$ is a $*$-subspace of a unital C*-algebra $\cA$, so
that $\cA/\cV$ is an abstract operator $*$-space. We show
that in this case
there exist a unital $*$-representation
 $(\cH, \pi)$ of $\cA$, a projection $P$ on $\cH$, and a Hermitian unitary
 operator $U$ on $\cH$ that commutes with the representation 
 $\pi$, such that when we define the linear $*$-map $\Psi$ from
 $\cA$ into $\cL(\cH)$ by
 \[
 \Psi(A) = PU\pi(A)P
 \]
 then $\Psi$ gives a completely isometric $*$-map from $\cA/\cV$
 onto a $*$-subspace of $\cL(\cH)$. In section \ref{secsys} we then
 briefly consider the case in which $\cV$ is an operator system,
 that is, $\cV$ is a $*$-subspace of $\cA$ that contains the
 identity element of $\cA$.
 
 Finally, in Section \ref{secsubalg} we discuss the situation in
 which $\cV$ is a unital C*-subalgebra, $\cB$, as described
 above.


\section{Quotients of operator spaces}
\label{secos}

We begin by reviewing here various facts about operator spaces that
we need.
Let $\cV$ be a vector space. For each natural number $n$ we let
$M_n(\cV)$ denote the vector space of $n \times n$ matrices
with entries in $\cV$. Let $\cA$ be a C*-algebra. Then $M_n(\cA)$
is a $*$-algebra in the evident way, and it has a unique C*-algebra
norm. We always view $M_n(\cA)$ as equipped with this norm. 
If $\cV$ is a subspace of $\cA$, then for 
each natural number $n$ we equip $M_n(\cV)$ with
the restriction to $M_n(\cV)$ of the
norm on $M_n(\cA)$. The resulting family of norms on all
these matrix spaces is called a matrix norm, and when $\cV$
is equipped with this family of norms it is called a ``concrete
operator space''. Ruan \cite{ER} found
axioms that characterize such families of norms. A family
of norms that satisfy Ruan's axioms  is called an ``operator-space
matrix norm''. 
A vector space equipped with an operator-space
matrix norm (but that is not assumed to
be a subspace of a C*-algebra) is called an ``abstract operator
space''. We will not need to use Ruan's axioms, because all
of the vector spaces that we consider will either be assumed to
be subspaces of C*-algebras, or will eventually be proved
to be (at least isomorphic to) such.  

If $\cV$ and $\cW$ are vector spaces and if $\phi$ is a linear
map from $\cV$ into $\cW$, then by entry-wise application $\phi$
determines a linear map, $\phi_n$, from $M_n(\cV)$
to $M_n(\cW)$ for each $n$. 
If $\cV$ and $\cW$ are each equipped with matrix
norms, then $\phi$ is said to be ``completely contractive'' if
the norm of each $\phi_n$ is no greater than 1, and $\phi$ is
said to be a ``complete isometry'' if each $\phi_n$ is an
isometry.

If $\cV$ is a closed subspace of an operator space $\cW$, so that 
$M_n(\cV)$ is a subspace of $M_n(\cW)$ for each $n$, then
for each $n$ we can equip
$M_n(\cW)/M_n(\cV)$ with the corresponding
quotient norm, thus obtaining a ``quotient matrix norm'' on $\cW/\cV$. 
Important perspective for us is given by the fact that
$\cW/\cV$ equipped with this quotient 
matrix norm is an abstract operator space
\cite{ER}. But again, in the end we will not actually have used this fact,
though we will use this terminology, as we do already 
in the next proposition.

The main technical step for all of the results of this paper
is given by the following proposition, which is closely related to
the GNS construction. Here we denote the Banach-space
dual of a Banach space $X$ by $X'$.

\begin{proposition}
\label{ccmap}
Let $\cA$ be a unital C*-algebra, let $\cV$ be a 
closed subspace of $\cA$, and equip $\cA/\cV$ with the 
corresponding quotient matrix norm
(so that $\cA/\cV$ is 
an abstract operator space). 
For a given natural number $n$ let there be given $\psi \in (M_n(\cA))'$ 
with $\psi(M_n(\cV)) = 0$ and $\|\psi\| = 1$. Then there
exist a unital $*$-representation, $(\cH, \pi)$, of $\cA$, and two projections,
$P$ and $Q$, in $\cL(\cH)$, each of rank no greater than $n$, such that when
we define the completely contractive map $\Psi : \cA \to \cL(\cH)$ by
\[
\Psi(A) = Q\pi(A)P
\]
for $A \in \cA$, then $\Psi(\cV) = 0$ and there exist two unit vectors, $\xi$ and $\eta$, 
in $\cH^{\oplus n}$ such that
\[
\psi(C) = \<\Psi_n(C)\xi, \eta\>
\]
for all $C \in M_n(\cA)$.
\end{proposition}

\begin{proof}
It is well-known that if $\cB$ is a unital C*-algebra and if $\th \in \cB'$ with $\|\th\| = 1$, then there
exist a unital $*$-representation, $(\rho, \cK)$, of $\cB$ and unit vectors $\xi^0$ 
and $\eta^0$ in $\cK$, such that $\th(B) = \<\rho(B)\xi^0, \eta^0\>$ for all $B \in \cB$.
See lemma 3.3 of \cite{R24} for a proof of this fact whose main tool is 
just the Jordan decomposition of a Hermitian linear functional
into the difference of two positive linear functionals. 
Accordingly, we can choose a unital $*$-representation $(\cK, \rho)$ of $M_n(\cA)$, 
and unit vectors 
$\xi^0$ and $\eta^0$ in $\cK$, such that 
\[
\psi(C) = \<\rho(C)\xi^0, \eta^0\>
\]
for all $C \in M_n(\cA)$.

Let $\{E_{jk}\}$ be the standard matrix-units for $M_n \subseteq M_n(\cA)$. 
Then $\rho(E_{11})$
is a projection in $\cL(\cK)$. Set $\cH = \rho(E_{11})\cK$. Define a unital 
$*$-representation
, $\pi$, of $\cA$ on $\cH$ by $\pi(A) = \rho(A\otimes E_{11})$ where here
we view $M_n(\cA)$ as $A \otimes M_n$. Then it is well-known and easily checked that
$(\cK, \rho)$ is unitarily equivalent to $(\cH^{\oplus n}, \pi_n)$, where by $\pi_n$ we mean the representation of $M_n(\cA)$ on $\cH^{\oplus n}$ defined by the
matrices
\[
\pi_n(C) = \{\pi(C_{jk})\}
\]
for $C \in M_n(\cA)$ and $C = \{C_{jk}\}$ with $C_{jk} \in \cA$,
and where the matrix $\{\pi(C_{jk})\}$ acts on $\cH^{\oplus n}$ in the evident
way. (This is, for example, essentially proposition 5ii of chapter I
of \cite{DvN}.) In particular, there will be unit vectors $\xi$ and $\eta$ in
$\cH^{\oplus n}$ such that
\[
\psi(C) = \<\pi_n(C)\xi, \eta\>
\]
for all $C\in M_n(\cA)$. 

Let $\xi = \{\xi_k\}$ and $\eta = \{\eta_j\}$ for $\xi_k, \eta_j \in \cH$. (Note
that the $\xi_k$'s are generally not orthogonal, and some may be 0,
and similarly for the $\eta_j$'s.) Then for $C = \{C_{jk}\}$ as above, we have
\[
\psi(C) = \<\pi_n(C)\xi, \eta\> = \Sigma_{jk}\<\pi(C_{jk})\xi_k, \eta_j\>  .
\]
Let $D \in \cV$, and for fixed $p$ and $q$ with $1\leq p, q \leq n$ let 
$C = D\otimes E_{pq}$, so that $C \in M_n(\cV)$. Then by assumption
on $\psi$
\[
0 = \psi(C) = \<\pi(D)\xi_q, \eta_p\>.
\]
Thus for all $p$ and $q$ we have
\[
\<\pi(\cV)\xi_q, \eta_p\> = 0 .
\]
Let $P$ and $Q$ be the projections onto, respectively, the linear spans of
$\{\xi_k\}$ and $\{\eta_j\}$. Thus $P$ and $Q$ are projections on $\cH$
of rank at most $n$. Furthermore, the fact that $\<\pi(D)\xi_q, \eta_p\> = 0$
for all $D \in \cV$ and all $p$ and $q$ implies that
\[
Q\pi(D)P = 0
\]
for all $D \in \cV$.

Define the linear mapping $\Psi$ from $\cA$ into $\cL(\cH)$ by
\[
\Psi(A) = Q\pi(A)P
\]
for all $A \in \cA$. Then it is standard and easily checked that
$\Psi$ is completely contractive. Of course, $\Psi(\cV) = 0$.
Furthermore, if we let $\Psi_n$ be the corresponding mapping
from $M_n(A)$ into $M_n(\cL(\cH))$, and if we let $P_n$ and $Q_n$
denote the diagonal $n \times n$ matrices with $P$, respectively
$Q$, in each diagonal entry, then 
\[
\Psi_n(C) = Q_n\pi_n(C)P_n
\]
for all $C \in M_n(\cA)$, where $\pi_n$ is as defined earlier in this proof. Note that 
$P_n\xi = \xi$ and $Q_n\eta = \eta$. Then as above
\[
\psi(C) = \<\pi_n(C)\xi, \eta\> = \<Q_n\pi_n(C)P_n\xi, \eta\> = \<\Psi_n(C)\xi, \eta\>
\]
for all $C \in M_n(\cA)$, as desired.
\end{proof}

We remark that if for each non-zero $\xi_k$ we let $P_k$ be the rank-one projection
with $\xi_k$ in its range, and if we define $Q_j$ similarly for $\eta_j$, then the above proposition can be reformulated in terms of the complete contractions $\Phi_{jk}(A) = Q_j\pi(A)P_k$. But this reformulation seems to be a bit more complicated.

For each natural number $n$ let $(M_n(\cV))^\perp$ denote the linear
subspace of $(M_n(\cA))'$ consisting of the linear functionals that take 
value 0 on $M_n(\cV)$. By the Hahn-Banach theorem, for each 
$C \in M_n(\cA)$ there is a $\psi \in (M_n(\cV))^\perp$ 
such that $\|\psi\| = 1$ and
$\psi(C) = \|C\|_{\cA/\cV}$, where $\|\cdot\|_{\cA/\cV}$ denotes the quotient
norm on $M_n(\cA)/M_n(\cV)$ pulled back to $M_n(\cA)$. 
Thus we can choose, in many ways, a
subset, $\cS_n^\cV$, of elements of $(M_n(\cV)^\perp$ of norm 1 such 
that for every $C \in M_n(\cA)$ we have
\[
 \|C\|_{\cA/\cV} = \sup\{|\psi(C)|: \psi \in \cS_n^\cV \}   . 
\]
For example, $\cS_n^\cV$ could consist of all elements $\psi$ of
$(M_n(\cV))^\perp$ of norm 1, or of a norm-dense subset thereof, or of the
set of extreme points of the unit ball of $(M_n(\cV))^\perp$. For each
such $\psi$ we obtain from the above proposition a representation
$(\cH^\psi, \pi^\psi)$ and projections $P^\psi$ and $Q^\psi$ on $\cH^\psi$,
and the corresponding completely contractive mapping $\Psi^\psi$
from $\cA$ into $\cL(\cH^\psi)$ defined by
\[
\Psi^\psi(A) = Q^\psi\pi^\psi(A)P^\psi  .
\]
Let $\cH^{\cV, n} = \bigoplus \{\cH^\psi : \psi \in \cS_n^\cV\}$, the Hilbert
space direct sum, and let 
$\pi^{\cV, n} = \bigoplus\{\pi^\psi : \psi \in \cS_n^\cV\}$ be the 
corresponding representation of $\cA$ on $\cH^{\cV, n}$. Let
$P^{\cV, n} = \bigoplus\{P^\psi:\psi \in \cS_n^\cV\}$, and define
$Q^{\cV, n}$ similarly. Then define $\Psi^{\cV, n}$ by
\[
\Psi^{\cV, n}(A) = Q^{\cV, n} \pi^{\cV, n}(A)P^{\cV, n}
\]
for all $A \in \cA$. Then from the requirements on $\cS_n^\cV$ it is clear that
for every $C\in M_n(\cA)$ we have
\[
\|C\|_{\cA/\cV} = \|\Psi_n^{\cV, n}(C)\|  .
\]

Now let $\cH^\cV = \bigoplus\{\cH^{\cV, n} : n \in \mathbb N\}$, let
$\pi^\cV = \bigoplus \{\pi^{\cV, n} : n \in \mathbb N\}$, and define projections
$P^\cV$ and $Q^\cV$ on $\cH^\cV$ similarly. Then from the
above considerations we see that we obtain:

\begin{theorem}
\label{thmconc}
Let $\cA$ be a unital C*-algebra, let $\cV$ be a norm-closed
subspace of $\cA$, and equip $\cA/\cV$ with the 
corresponding quotient matrix norm. 
Then the constructions above provide a
unital $*$-representation $(\cH^\cV, \pi^\cV)$ of $\cA$,
and projections $P^\cV$ and $Q^\cV$ on $\cH^\cV$, such that
the linear mapping $\Psi^\cV$ from $\cA$ to $\cL(\cH^\cV)$ defined
by
\[
\Psi^\cV(A) = Q^\cV\pi^\cV(A)P^\cV
\]
gives a complete isometry from $\cA/\cV$ into $\cL(\cH^\cV)$.
\end{theorem}


\section{Quotients of operator $*$-spaces}
\label{secstar}

My principal aim is to understand quotients of the form $\cA/\cB$ where
$\cA$ is a C*-algebra and $\cB$ is a C*-subalgebra of $\cA$. 
But both $\cA$ and $\cB$ are stable under $^*$, and so we
will consider first quotients under just that requirement.

\begin{definition}
By a \emph{concrete operator $*$-space} we mean a subspace $\cW$ of
some C*-algebra $\cA$ that is stable under $^*$, that is, if $A \in \cW$
then $A^* \in \cW$. 
\end{definition}

If $\cW$ is a $*$-stable subspace of some C*-algebra $\cA$, then 
$M_n(\cW)$ is a $*$-stable subspace of $M_n(\cA)$ for 
each natural number $n$, and the restriction to $M_n(\cW)$ of the
norm on $M_n(\cA)$ will be a $*$-norm in the
sense that $\|C^*\|_n = \|C\|_n$ for all $C \in M_n(\cW)$. 

By a \emph{vector $*$-space} we mean (definition 3.1 of \cite{PlT}) a vector space
$\cW$ over $\mathbb C$ that is equipped with a $*$-operation
(i.e. involution) satisfying the usual properties. Then $M_n(\cW)$
is also canonically a vector $*$-space where $(C^*)_{jk} = (C_{kj})^*$
for $C = \{C_{jk}\}$ as one would expect.

\begin{definition}
Let $\cW$ be a vector $*$-space. By a \emph{matrix $*$-norm} on
$\cW$ we mean a matrix norm $\{\|\cdot\|_n\}$ on $\cW$ such that
each $\|\cdot\|_n$ is a $*$-norm. By an \emph{abstract operator
$*$-space} we mean a vector $*$-space that is equipped with a 
matrix $*$-norm that satisfies Ruan's axioms.
\end{definition}

Let $\cW$ be an operator $*$-space, and let $\cV$ be a closed
$*$-subspace of $\cW$ (that is, $\cV$ is stable under the involution
on $\cW$). Then the involution on $\cW$ gives an involution on
$\cW/\cV$ in the evident way, so that $\cW/\cV$ is a vector $*$-space.
Then the quotient norm from each $\|\cdot\|_n$ will be a $*$-norm.
In this way $\cW/\cV$ is an abstract operator $*$-space.

We will now show that if $\cW$ is a concrete operator $*$-space
then we can use the results of the previous section to obtain a 
completely isometric $*$-representation of $\cW/\cV$ as a
concrete operator $*$-space. As in the previous section, it suffices 
to do this for the case in which $\cW$ is a unital C*-algebra
$\cA$. So we now treat that case. Then by Theorem \ref{thmconc}
there exist a $*$-representation $(\cH, \pi)$ of $\cA$ and
projections $P$ and $Q$ in $\cL(\cH)$ such that the linear map
$\Psi: \cA \to \cL(\cH)$ defined by
\[
\Psi(A) = Q\pi(A)P
\]
gives a complete isometry from $\cA/\cV$ into $\cL(\cH)$.
Define $\Psi^*$ by $\Psi^*(A) = (\Psi(A^*))^*$ as usual. 
Notice that $\Psi^*(A) = P\pi(A)Q$ for all $A \in \cA$, and
that $\Psi^*(\cV) = 0$. 
Define $\Phi: \cA \to \cL(\cH \oplus \cH)$ by
\[
\Phi(A) =\begin{pmatrix}  0 & P\pi(A)Q \\
                                          Q\pi(A)P & 0       \end{pmatrix}   .
                                          \]
Then it is easily seen that $\Phi$ is a $*$-map. Clearly $\Phi$
is contractive, and it is a complete isometry since $\Psi$ is.
We can rewrite $\Phi$ as
\[
\Phi(A) = \begin{pmatrix}  P & 0 \\
                                          0 & Q       \end{pmatrix} 
               \begin{pmatrix}  0 & \pi(A) \\
                                          \pi(A) & 0       \end{pmatrix}   
                \begin{pmatrix}  P & 0 \\
                                          0 & Q       \end{pmatrix}    ,                                                  
\]
and we see that ($ \begin{smallmatrix}  P & 0 \\
                                          0 & Q       \end{smallmatrix} $) is itself
a projection. But 
($  \begin{smallmatrix}  0 & \pi(A) \\
                                          \pi(A) & 0       \end{smallmatrix}   $)
does not quite give a $*$-representation of $\cA$. It is thus more 
attractive to rewrite $\Phi$ as
\[                                        
\Phi(A) = \begin{pmatrix}  P & 0 \\
                                          0 & Q       \end{pmatrix} 
                \begin{pmatrix}  0 & 1 \\
                                          1 & 0       \end{pmatrix}                            
               \begin{pmatrix}   \pi(A) & 0 \\
                                          0 & \pi(A)       \end{pmatrix}   
                \begin{pmatrix}  P & 0 \\
                                          0 & Q       \end{pmatrix}    ,                                                  
\]
and to notice that ($ \begin{smallmatrix}  0 & 1 \\
                                          1 & 0       \end{smallmatrix}    $)
is a Hermitian unitary on $\cH \oplus \cH$ that commutes with
the representation $\pi \oplus \pi$ of $\cA$. This puts $\Phi$  
into the ``commutant representation'' form given in theorems 2.2,
2.9 and 2.10
of \cite{PlS}. On changing the meaning of the various symbols
$\cH$, $\pi$, $P$, etc, we thus obtain:

\begin{theorem}
\label{thmstar}
 Let $\cA$ be a unital C*-algebra, let $\cV$ be a closed 
 $*$-subspace of $\cA$, and equip $\cA/\cV$ with the 
corresponding quotient matrix norm (so that $\cA/\cV$ is an operator
 $*$-space). Then there exist a unital $*$-representation
 $(\cH, \pi)$ of $\cA$, a projection $P$ on $\cH$, and a Hermitian
 unitary
 operator $U$ on $\cH$ that commutes with the representation 
 $\pi$, such that the linear $*$-map $\Psi$ from
 $\cA$ into $\cL(\cH)$ defined by
 \[
 \Psi(A) = PU\pi(A)P  ,
 \]
 gives a completely isometric $*$-map from $\cA/\cV$
 onto a $*$-subspace of $\cL(\cH)$.
 \end{theorem}                      
 
 Notice that we can cut down to the closure of $\pi(\cA)P\cH$, that is, 
 we can assume that $\pi(\cA)P\cH$ is dense in $\cH$.       
 
 Let $E$ and $F$ be the projections onto the two eigensubspaces of
 $U$, so that $U = E - F$ and $E + F = I_\cH$. Then we can decompose $\Psi$ as
 \[
     \Psi(A) = PE\pi(A)EP  \ - \ PF\pi(A) FP  .
 \]            
 The two terms on the right give completely positive maps. Thus this
 decomposition can be viewed as an analogue for $\Psi$ of the
 Jordan decomposition of a signed measure. But note that $PE$
 is not in general a projection.

\section{Quotients of operator systems}
\label{secsys}

In this section we consider quotients of operator systems.
As before, it suffices for us to consider $\cV$ as a 
subspace of a C*-algebra $\cA$.
Thus we assume that $\cV$ is an operator system in $\cA$, that
is, that $\cV$ is a closed $*$-subspace that contains the identity element,
$1_\cA$, of $\cA$. On applying Theorem \ref{thmstar}, with the notation
used there, we obtain a completely isometric embedding of $\cA/\cV$
into $\cL(\cH)$ given by a map $\Psi: \cA \mapsto \cL(\cH)$ defined by
 \[
 \Psi(A) = PU\pi(A)P  .
 \]
 The extra information that we obtain from having $1_\cA \in \cV$ is that
 \[
 0 = \Psi(1_\cA) = PUP.
\]
From this and the fact that $U$ commutes with $\pi(A)$ we see that 
\begin{eqnarray}
\Psi(A)  & = & P\pi(A)UP \ - \ PUP\pi(A)   \nonumber   \\ 
& = & \ P[\pi(A), UP] \ = \ PU[\pi(A), P]. 
\label{eqder}
\end{eqnarray}
Let $X = 2P-I$, so that $X$ is a Hermitian unitary. Then
it follows that we can express $\Psi$ by
\[
\Psi(a) = - (1/2)PU[X, \pi(A)].
\] 
We can equally well express $\Psi$ by
\begin{eqnarray}
\Psi(A)  & = & PU\pi(A)P \ - \ \pi(A)PUP   \nonumber   \\ 
& = & \ [PU, \pi(A)]P \ = \ [P, \pi(A)]UP.   \nonumber
\end{eqnarray}
On adding the third term of this equation to that of equation \eqref{eqder}
we obtain
\[
\Psi(A) = P[ \ [P, U]/2, \pi(A)]P    .
\]
Set $Z = [P, U] = [2P-I, U]/2 = [X, U]/2$. Clearly $Z^* = -Z$ and $\|Z\| \leq 1$.
Furthermore, $[U, P^2] = [U, P]P + P[U, P]$ so that $PZ = Z(I-P)$. 
Thus we obtain:

\begin{theorem}
\label{thmsys}
 Let $\cA$ be a unital C*-algebra, let $\cV$ be an operator system
 in $\cA$, and equip $\cA/\cV$ with the 
corresponding quotient matrix norm (so that $\cA/\cV$ is an abstract operator
 $*$-space). Then there exist a unital $*$-representation
 $(\cH, \pi)$ of $\cA$, a projection $P$ on $\cH$, and an
 operator $Z$ on $\cH$ satisfying $Z^* = -Z$ and $\|Z\| \leq 1$
 and $PZ = Z(I-P)$,
 such that the linear $*$-map $\Psi$ from
 $\cA$ into $\cL(\cH)$ defined by
 \[
 \Psi(A) = (1/2)P[Z, \pi(A)]P  
 \]
 gives a completely isometric $*$-map from $\cA/\cV$
 onto a $*$-subspace of $\cL(\cH)$.
 \end{theorem}                      

We remark that a quite different type of quotient involving operator systems, in which one
wants the quotient of an operator system by the kernel of a completely positive map 
to be an operator system, is studied in \cite{KPT, FrP}.


\section{Quotients by C*-subalgebras}
\label{secsubalg}

In this section we assume that $\cA$ is a unital C*-algebra and
that $\cB$ is a unital C*-subalgebra of $\cA$ (so $1_\cA \in \cB$).
Since $\cB$ is, in particular, a $*$-subspace of $\cA$, Theorem 
\ref{thmstar} is applicable, and, with the notation used there, we have
a completely isometric embedding of $\cA/\cB$ into
$\cL(\cH)$ given by the map $\Psi: \cA \to \cL(\cH)$ defined by
\[
\Psi(A) = PU\pi(A)P   .
\]

Now let $\hat P$ be the projection onto the closed linear span
of $\pi(\cB)P\cH$. Since the range of $\hat P$ is $\pi(\cB)$-invariant,
$\hat P$ commutes with $\pi(B)$ for all $B \in \cB$. Because
$1_\cA \in \cB$, the range of $\hat P$ contains $P\cH$, and
so $\hat P \geq P$. From the fact that $0 = \Psi(\cB) = PU\pi(\cB)P$
and that $\cB$ is an algebra it
is easily seen that $PU\pi(\cB)\hat P = 0$. On taking adjoints, we
have $\hat P U\pi(\cB)P = 0$, and so in the same way as above
we have $\hat P U\pi(\cB)\hat P = 0$. Define 
$\hat \Psi: \cA \to \cL(\cH)$ by
\[
\hat \Psi(A) = \hat P U \pi(A)\hat P   .
\]
Clearly $\hat \Psi$ is completely contractive and $\hat \Psi(\cB) = 0$. From the fact 
that $\hat P \geq P$ we see that $\|\hat \Psi (A)\| \geq \|\Psi (A)\|$
for all $A \in \cA$, and it is easily seen that in fact
$\|\hat \Psi_n (C)\| \geq \|\Psi_n (C)\|$ for all natural numbers $n$
and all $C \in M_n(\cA)$. Since $\Psi$ gives a complete isometry
from $\cA/\cB$ into $\cL(\cH)$, it follows that $\hat \Psi$ does also.

Now let $X = 2\hat P \ - \ I$. Then $X$ is a Hermitian unitary
in $\cL(\cH)$ that commutes with $\pi(B)$ for every
$B \in \cB$. Notice that because $1_\cA \in \cB$ we have
$\hat P  U \hat P \ = \ 0$. Then much as in the calculation for
equation \eqref{eqder} we find that
\[
\hat \Psi(A) = -(1/2)\hat P U[X \ , \ \pi(A)] .
\]
It follows that $\|[X, \ \pi(A)]\| \geq 2\|\hat \Psi(A)\|$ for
all $A \in \cA$.
Define a derivation, $\Theta$, from $\cA$ into $\cL(\cH)$ by
\[
\Theta(A) \ = \ (1/2)[X, \ \pi(A)]   
\]
for all $A \in \cA$. Then $\|\Theta(A)\| \geq \|\hat \Psi(A)\|$
for all $A \in \cA$.
Furthermore, $\Theta$ is completely contractive. To see this,
notice that it is the composition of $\pi$ with a corner of
the completely positive contraction that sends 
($ \begin{smallmatrix}  a & c \\
                             b & d       \end{smallmatrix}    $) in $M_2(\cA)$ to
\[
(1/2) \begin{pmatrix}  I & 0 \\
                                    0 & X       \end{pmatrix}    
 \begin{pmatrix}  a & c \\
                             b & d       \end{pmatrix}    
 \begin{pmatrix}  I & 0 \\
                           0 & X       \end{pmatrix}    
 \ + \ (1/2) \begin{pmatrix}  -X & 0 \\
                                    0 & I       \end{pmatrix}    
\begin{pmatrix}  a & c \\
                             b & d       \end{pmatrix}    
\begin{pmatrix}  -X & 0 \\
                              0 & I       \end{pmatrix}.                                                                                
\]
A slight modification of the calculations done a few lines above shows easily
that
$\|\Theta_n (C)\| \geq \|\hat \Psi_n (C)\|$ for all natural numbers $n$
and all $C \in M_n(\cA)$.
Notice that $\Theta(B) = 0$ for all $B \in \cB$
because $\hat P$ commutes with all of the elements of $\pi(\cB)$.
Since $\hat \Psi$ gives a complete isometry from $\cA/\cB$
into $\cL(\cH)$, it follows that $\Theta$ does also. Notice that
if we replace $X$ by $iX$ then $\Theta$ is a $*$-map. We
have thus obtained:

\begin{theorem}
\label{thmder}
Let $\cA$ be a unital C*-algebra, and let $\cB$ be a unital C*-subalgebra
of $\cA$ (so $1_\cA \in \cB$). Then there exist a unital $*$-representation
$(\cH, \pi)$ of $\cA$, and a Hermitian unitary operator $X$ on $\cH$
that commutes with $\pi(B)$ for all $B \in \cB$, such that the
derivation $\Theta$ from $\cA$ into $\cL(\cH)$ defined by
\[
\Theta(A) \ = \ (1/2)[iX, \ \pi(A)]   
\]
gives a completely isometric $*$-map from $\cA/\cB$ into
$\cL(\cH)$.
\end{theorem}

This theorem is a strengthening of corollary 3.4 of \cite{R24},
and its proof is in part motivated by the proof of theorem 3.1
of \cite{R24}.

                                        
                                        
\def\dbar{\leavevmode\hbox to 0pt{\hskip.2ex \accent"16\hss}d}
\providecommand{\bysame}{\leavevmode\hbox to3em{\hrulefill}\thinspace}
\providecommand{\MR}{\relax\ifhmode\unskip\space\fi MR }
\providecommand{\MRhref}[2]{%
  \href{http://www.ams.org/mathscinet-getitem?mr=#1}{#2}
}
\providecommand{\href}[2]{#2}

\end{document}